\documentclass[11pt,twoside]{article}
\usepackage{acta-m,euler}

\usepackage{epic}
\usepackage{eepic}

\title{%
Reconstruction of complete interval tournaments. II.
}

\newcommand{\vol}{\textbf{2}, 1 (2010) 47--71}   
\newcommand{\amss}{\amssubj{%
05C20, 68C25
}}     
\newcommand{\keyw}{\keywords{%
score sequences, tournaments, efficiency of algorithms
}}

\begin{document}
\setcounter{page}{47}
\maketitle

\oneauthor{%
Antal Iv\'anyi}{%
E\"otv\"os Lor\'and University,

 Department of Computer Algebra

 1117 Budapest, P\'azm\'any P. s\'et\'any 1/C.}{%
tony@compalg.inf.elte.hu}

\short{%
A. Iv\'anyi
}{%
Reconstruction of complete interval tournaments. II.
}

\begin{abstract}
Let $a, \ b \ (b \geq a)$ and $n \ (n \geq 2)$ be nonnegative integers and let $\mathcal{T}(a,b,n)$ be the set of such generalised
tournaments, in which every pair of distinct players is connected at most with $b$, and at least with $a$ arcs. In \cite{Ivanyi2009}
we gave a necessary and sufficient condition to decide whether a given sequence of nonnegative integers
$D = (d_1, d_2, \ldots, d_n)$ can be realized as the out-degree sequence of a $T \in \mathcal{T}(a,b,n)$. Extending the results
of \cite{Ivanyi2009} we show that for any sequence of nonnegative integers $D$ there exist $f$ and $g$ such that some element
$T \in \mathcal{T}(g,f,n)$ has $D$ as its out-degree sequence, and for any $(a,b,n)$-tournament $T'$ with the same out-degree
sequence $D$ hold $a\leq g$ and $b\geq f$. We propose a $\Theta(n)$ algorithm to determine $f$ and $g$ and an $O(d_n n^2)$
algorithm to construct a corresponding tournament $T$.
\end{abstract}

\section{Introduction}
Let $a, \ b \ (b \geq a)$ and $n \ (n \geq 2)$ be nonnegative integers and let $\mathcal{T}(a,b,n)$ be the set of such generalised
tournaments, in which every pair of distinct players is connected at most with $b$, and at least with $a$ arcs. The elements of
$\mathcal{T}(a,b,n)$ are called $(a,b,n)$\textit{-tournaments}. The vector $D = (d_1, d_2, \ldots, d_n)$ of the out-degrees of
$T \in \mathcal{T}(a,b,n)$ is called \textit{the score vector} of $T$. If the elements of $D$ are in nondecreasing order,
then $D$ is called \textit{the score sequence} of $T$.

An arbitrary vector $D = (d_1, d_2, \ldots, d_n)$ of nonnegative integers is called \textit{graphical vector,} iff there exists a
loopless multigraph whose degree vector is $D$, and $D$ is called \textit{digraphical vector} (or \textit{score vector}) iff there
exists a loopless directed multigraph whose out-degree vector is $D$.

A nondecreasingly ordered graphical vector is called \textit{graphical sequence}, and a nondecreasingly ordered digraphical
vector is called \textit{digraphical sequence} (or \textit{score sequence}).

The number of arcs of $T$ going from player P$_i$ to player P$_j$ is denoted by $m_{ij} \ (1 \leq i,j \leq n)$, and
the matrix $\mathcal{M} = [1. \ .n,1 . \ .n]$ is called \textit{point matrix} or \textit{tournament matrix} of $T$.

In the last sixty years many efforts were devoted to the study of both types of vectors, resp. sequences. E.g. in the papers
\cite{Berge1976,Erdos1960,Frank1976,Frank1980,Frank1996,Frank2003,Fulkerson1960,Griggs2004,Hakimi1962,Hakimi1974,Havel1955,
Katona1966,Patrinos1976,Senior1951,Sierksma1991,Szekely1992,
Tripathi2003,Wang2008} the graphical sequences, while in the papers \cite{Acosta2003,Avery1991,Bang1979,Belica2000,Berge1976,Brualdi2001,
Ford1962,Gervacio1988,Gervacio1993,Griggs1999,Guiduli1998,Hakimi1965,Hemasinha2003,Kleitman1976,Kleitman1973,Kleitman1981,
Landau1953,Mahmoodian1978,McKay1996,Moon1962,Moon1963,Moon1968,Narayana1964,Narayana1971,Ore1956,
Pecsy2000,Reid1996,Reid1998,Ruskey1994,vandenBrink2003,Stockmeyer1977,Stockmeyer2009,Wang1973,
Winston1983,Zhou2000} the score sequences were discussed.

Even in the last two years many authors investigated the conditions, when $D$ is graphical  (e.g.
\cite{Barrus2008,Boesch1976,Brualdi2009,Chung2008,Frank20081,Frank20082,Frank20083,Frank2009,Hu2009,Hulett2008,Jordon2009,Kim2009,
Klivans2008,Klivans2009,Meierling2009,Pirzada2009,Rodseth2009,Tripathi2008,Tripathi2008a,vandenBrink2009,Volkmann2009,Yin2008})
or digraphical (e.g. \cite{Beasley2009,
Harborth1982,Ivanyi2009,Kemnitz1997,Knuth2008,Lovasz2008,Nabyjev2008,Palvolgyi2009,Pirzada20082di,Pirzada2008imbalances,
Pirzada2008Inequalities,Pirzada2008oriented,Pirzada2008orientedhyper,Ryser1964,
Stockmeyer2009,Thomassen1981,Zhou2008}).

In this paper we deal only with directed graphs and usually follow the terminology used by K. B. Reid \cite{Reid1998,Reid2004}.
If in the given context $a, \ b$ and $n$ are fixed or non important, then we speak simply on \textit{tournaments} instead of
generalised or $(a,b,n)$-tournaments.

We consider the loopless directed multigraphs as generalised tournaments, in which the number of arcs from vertex/player
P$_i$ to vertex/player P$_j$ is denoted by $m_{ij}$, where $m_{ij}$ means the number of points won
by player P$_i$ in the match with player P$_j$.

The first question: how one can characterise the set of the score sequences of the $(a,b,n)$-tournaments. Or, with another words,
for which sequences $D$ of nonnegative integers does exist an $(a,b,n)$-tournament whose out-degree sequence is $D$.
The answer is given in Section \ref{section-exist}.

If $T$ is an $(a,b,n)$-tournament with point matrix $\mathcal{M}  = [1. \ .n,1 . \ .n]$, then let $E(T), \ F(T)$ and $G(T)$ be
defined as follows: $E(T) = \max_{1\leq i,j\leq n} m_{ij}$, $F(T) = \max_{1\leq i < j\leq n} (m_{ij} + m_{ji})$, and
$g(T) = \min_{1\leq i < j\leq n} (m_{ij} + m_{ji})$. Let $\Delta(D)$ denote the set of all tournaments having $D$ as
out-degree sequence, and let $e(D), \ f(D)$ and $g(D)$ be defined as follows:
$e(D) = \{\min \ E(T) \ | \ T \in \Delta(D)\}$, $f(D) = \{\min \ F(T) \ | \ T \in \Delta(D)\}$, and
$g(D) = \{\max \ G(T) \ | \ T \in \Delta(D)\}$. In the sequel we use the short notations $E, \ F, \ G, \ e, \ f, \ g,$
and $\Delta$.

Hulett et al. \cite{Hulett2008,Will2004}, Kapoor et al. \cite{Kapoor1977}, and Tripathi et al. \cite{Tripathi2006,Tripathi2008}
investigated the construction problem of a minimal size graph having a prescribed degree set \cite{Reid1978,Yao1989}.
In a similar way we follow a mini-max approach formulating the following questions: given a sequence $D$ of nonnegative integers,
\begin{itemize}
\item How to compute $e$ and how to construct a tournament $T \in \Delta$ characterised by $e$? In Section
\ref{section-e} a formula to compute $e$, and an algorithm to construct a corresponding tournament are presented.

\item How to compute $f$ and $g$? In Section \ref{section-fg} an algorithm to compute $f$ and $g$ is described.

\item How to construct a tournament $T \in \Delta$ characterised by $f$ and $g$? In Section \ref{section-construct}
an algorithm to construct a corresponding tournament is presented and analysed.
\end{itemize}

We describe the proposed algorithms in words, by examples and by the pseudocode used in \cite{Cormen2009}.

Researchers of these problems often mention different applications, e.g. in biology \cite{Landau1953}, chemistry Hakimi
\cite{Hakimi1962}, and Kim et al. in networks \cite{Kim2009}.

\section{Existence of a tournament with arbitrary degree sequence\label{section-exist}}

Since the numbers of points $m_{ij}$ are not limited, it is easy to construct a $(0,d_n,n)$-tournament for any $D$.

\begin{lemma} If $n \geq 2$, then for any\label{lemma-exist} vector of nonnegative integers $D = (d_1,$ $d_2, \ldots, d_n)$
there exists a loopless directed multigraph $T$ with out-degree vector $D$ so, that $E \leq d_n$.
\end{lemma}
\begin{proof} Let $m_{n1} = d_n$ and  $m_{i,i+1} = d_i$ for $i = 1, 2, \ldots, n - 1$, and let the remaining $m_{ij}$ values
be equal to zero.
\end{proof}

Using weighted graphs it would be easy to extend the definition of the $(a,b,n)$-tournaments to allow \textit{arbitrary real values} of
$a$, $b$, and $D$. The following algorithm \textsc{Naive-Construct} works without changes also for input consisting of real numbers.

We remark that Ore in 1956 \cite{Ore1956} gave the necessary and sufficient conditions of the existence of a tournament
with prescribed in-degree and out-degree vectors. Further Ford and Fulkerson \cite[Theorem11.1]{Ford1962}
published in 1962 necessary and sufficient conditions of the existence of a tournament having prescribed lower and upper
bounds for the in-degree and out-degree of the vertices. They results also can serve as basis of the existence of a tournament
having arbitrary out-degree sequence.

\subsection{Definition of a naive reconstructing algorithm}

Sorting of the elements of $D$ is not necessary.

\textit{Input}. $n $: the number of players $(n \geq 2)$; \\
$D = (d_1, d_2, \ldots, d_n)$: arbitrary sequence of nonnegative integer numbers.

\textit{Output}. $\mathcal{M} = [1. \ .n,1. \ . n]$: the point matrix of the reconstructed tournament.

\textit{Working variables}. $i, \ j$: cycle variables.

\medskip
\noindent \textsc{Naive-Construct}$(n,D)$
\vspace{-2mm}
\begin{tabbing}%
199 \= xxx\=xxx\=xxx\=xxx\=xxx\=xxx\=xxx\=xxx \+ \kill
\hspace{-7mm}01 \textbf{for} \= $i \leftarrow 1$ \textbf{to} $n$ \\
\hspace{-7mm}02              \> \textbf{for} \= $j \leftarrow 1$ \textbf{to} $n$ \\
\hspace{-7mm}03              \>              \> \textbf{do} $m_{ij} \leftarrow 0$ \\
\hspace{-7mm}04 $m_{n1} \leftarrow d_n$ \\
\hspace{-7mm}05 \textbf{for} \= $i \leftarrow 1$ \textbf{to} $n - 1$ \\
\hspace{-7mm}06              \> \textbf{do} $m_{i,i+1} \leftarrow d_i$ \\
\hspace{-7mm}07 \textbf{return} $\mathcal{M}$
\end{tabbing}

The running time of this algorithm is $\Theta(n^2)$ in worst case (in best case too). Since the point
matrix $\mathcal{M}$ has $n^2$ elements, this algorithm is asymptotically optimal.

\section{Computation of $e$\label{section-e}}

This is also an easy question. From here we suppose that $D$ is a nondecreasing sequence of nonnegative integers, that is
$0 \leq d_1$ $\leq d_2 \leq \ldots \leq d_n$. Let $h = \lceil d_n/(n - 1) \rceil $.

Since $\Delta(D)$ is a finite set for any finite score vector $D$, $e(D) = \min\{E(T) | T \in \Delta(D)\}$ exists.

\medskip
\begin{lemma} If $n \geq 2$, then for\label{lemma-e} any sequence $D = (d_1, d_2,\ldots,d_n)$ there exists
a $(0,b,n)$-tournament $T$ such that
\begin{equation}
E \leq h \qquad \hbox{and} \quad b \leq 2h,\label{equation-e}
\end{equation}
and $h$ is the smallest upper bound for $e$, and $2h$ is the smallest possible upper bound for $b$.
\end{lemma}
\begin{proof}
If all players gather their points in a uniform as possible manner, that is
\begin{equation}
\max_{1 \leq j \leq n} m_{ij} - \min _{1 \leq j \leq n, \ i \neq j} m_{ij} \leq 1 \quad
\hbox{ for } i = 1, \ 2, \ \ldots, \ n, \label{equation-uniform}
\end{equation}
then we get $E \leq h$, that is the bound is valid. Since player P$_n$ has to gather $d_n$ points, the pigeonhole
principle \cite{Bege2006,Dossey2001,Jarai2009} implies $E \geq h$, that is the bound is not improvable. $E \leq h$ implies
$\max_{1 \leq i < j \leq n} m_{ij} + m_{ji} \leq 2h$. The score sequence $D = (d_1,d_2,\ldots,d_n) = (2n(n -1),2n(n - 1), \ldots,
2n(n - 1))$ shows, that the upper bound $b \leq 2h$ is not improvable.
\end{proof}
\begin{corollary} If $n \geq 2$, then for\label{corollary-e} any sequence $D = (d_1, d_2, \ldots ,d_n)$ holds
$e(D) = \lceil d_n/(n - 1) \rceil$.
\end{corollary}
\begin{proof} According to Lemma \ref{lemma-e} $h = \lceil d_n/(n - 1) \rceil$ is the smallest upper bound for $e$.
\end{proof}

\subsection{Definition of a construction algorithm}

The following algorithm constructs a $(0,2h,n)$-tournament $T$ having
$E \leq h$ \ for any $D$.

\textit{Input}. $n $: the number of players $(n \geq 2)$; \\
$D = (d_1, d_2, \ldots, d_n)$: arbitrary sequence of nonnegative integer numbers.

\textit{Output}. $\mathcal{M} = [1. \ .n,1. \ . n]$: the point matrix of the tournament.

\textit{Working variables}. $i, \ j, \ l$: cycle variables; \newline
$k$: the number of the ''larger parts" in the uniform distribution of the points.

\medskip
\noindent \textsc{Pigeonhole-Construct}$(n,D)$
\vspace{-2mm}
\begin{tabbing}%
199 \= xxx\=xxx\=xxx\=xxx\=xxx\=xxx\=xxx\=xxx \+ \kill
\hspace{-7mm}01 \textbf{for} \= $i \leftarrow 1$ \textbf{to} $n$ \\
\hspace{-7mm}02              \> \textbf{do} \= $m_{ii} \leftarrow 0$ \\
\hspace{-7mm}03              \>             \> $k \leftarrow d_i - (n - 1)\lfloor d_i/(n - 1) \rfloor$ \\
\hspace{-7mm}04              \> \textbf{for} \= $j \leftarrow 1$ \textbf{to} $k$ \\
\hspace{-7mm}05              \>              \> \textbf{do} \= $l \leftarrow i + j \ (\hbox{mod} \ n)$ \\
\hspace{-7mm}06              \>              \>             \> $m_{il} \leftarrow \lceil d_n/(n - 1) \rceil$ \\
\hspace{-7mm}07              \> \textbf{for} \= $j \leftarrow k + 1$ \textbf{to} $n - 1$ \\
\hspace{-7mm}08              \>              \> \textbf{do} \= $l \leftarrow i + j \ (\hbox{mod} \ n)$ \\
\hspace{-7mm}09              \>              \>             \> $m_{il} \leftarrow \lfloor d_n/(n - 1) \rfloor$ \\
\hspace{-7mm}10 \textbf{return} $\mathcal{M}$
\end{tabbing}

The running time of \textsc{Pigeonhole-Construct} is $\Theta(n^2)$ in worst case (in best case too). Since the point matrix
$\mathcal{M}$ has $n^2$ elements, this algorithm is asymptotically optimal.

\section{Computation of $f$ and $g$\label{section-fg}}

Let $S_i \ (i = 1, \ 2, \ \ldots, \ n)$ be the sum of the first $i$ elements of $D$, $B_i \ (i = 1, \ 2, \ \ldots, \ n)$
be the binomial coefficient $n(n - 1)/2$.
Then the players together can have $S_n$ points only if $fB_n \geq S_n$. Since the score of player P$_n$ is $d_n$,
the pigeonhole principle implies $f \geq \lceil d_n/(n - 1) \rceil$.

These observations result the following lower bound for $f$:
\begin{equation}
f \geq \max \left (\left \lceil \frac{S_n}{B_n} \right  \rceil, \left \lceil \frac{d_n}{n - 1}
             \right  \rceil \right).\label{equation-flowerbound}
\end{equation}

If every player gathers his points in a uniform as possible manner then
\begin{equation}
f \leq 2 \left \lceil \frac{d_n}{n - 1} \right \rceil.\label{equation-fupperbound}
\end{equation}

These observations imply a useful characterisation of $f$.

\begin{lemma} If $n \geq 2$, then for\label{lemma-fgbounds} arbitrary sequence $D = (d_1, d_2, \ldots, d_n)$ there exists
a $(g,f,n)$-tournament having $D$ as its out-degree sequence and the following bounds for $f$ and $g$:
\begin{equation}
\max \left ( \left \lceil \frac{S}{B_n} \right \rceil, \left \lceil \frac{d_n}{n - 1} \right
\rceil \right) \leq f \leq 2 \left \lceil \frac{d_n}{n - 1} \right \rceil,\label{equation-flemma}
\end{equation}
\begin{equation}
0 \leq g \leq f.\label{equation-glemma}
\end{equation}
\end{lemma}
\begin{proof}
(\ref{equation-flemma}) follows from (\ref{equation-flowerbound}) and (\ref{equation-fupperbound}), (\ref{equation-glemma})
follows from the definition of $f$.
\end{proof}

It is worth to remark, that if $d_n/(n - 1)$ is integer and the scores are identical, then the lower and upper bounds in
(\ref{equation-flemma}) coincide and so Lemma \ref{lemma-fgbounds} gives the exact value of $F$.

In connection with this lemma we consider three examples.
If $d_i = d_n = 2c(n - 1) \ (c > 0, \ i = 1, \ 2, \ \ldots, \ n - 1)$, then
$d_n/(n - 1) = 2c$ and $S_n/B_n = c$, that is $S_n/B_n$ is twice larger than $d_n/(n - 1)$. In the other extremal case, when
$d_i = 0 \ (i = 1, \ 2, \ \ldots, n - 1)$ and $d_n = cn(n - 1) > 0$, then $d_n/(n - 1) = cn$, $S_n/B_n = 2c$, so $d_n/(n - 1)$ is
$n/2$ times larger, than $S_n/B_n$.

If $D = (0, 0, 0, 40, 40,40)$, then Lemma \ref{lemma-fgbounds} gives the bounds $8 \leq f \leq 16$.
Elementary calculations show that Figure \ref{figure-1} contains the solution with minimal $f$, where $f = 10$.

\begin{figure}[!h]
\begin{center}
\begin{tabular}{|c|c|c|c|c|c|c|c|}  \hline
Player/Player &  P$_1$ & P$_2$ & P$_3$ & P$_4 $ & P$_5$ & P$_5$ &  Score \\ \hline
P$_1$         & ---    &  $0$  & $0$   &    $0$ &  $0$  & $0$   &   $0$   \\ \hline
P$_2$         & $0$    &  ---  & $0$   & $0$    & $0$   & $0$   &   $0$  \\ \hline
P$_3$         & $0$    & $0$   &  ---  &  $0$   &  $0$  & $0$   &   $0$  \\ \hline
P$_4$         &$10$   & $10$ & $10$ &  ---   & $5$  & $5$  & $40$\\ \hline
P$_5$         &$10$   & $10$ & $10$ & $5$   &  ---  & $5$  & $40$ \\ \hline
P$_6$         &$10$   & $10$ & $10$ & $5$   &  $5$ &  ---  & $40$ \\ \hline
\end{tabular} \\
\caption{Point matrix of a $(0,10,6)$-tournament with $f = 10$ for $D = (0,0,0,40,40,40)$.\label{figure-1}}
\end{center}
\vspace{-2mm}
\end{figure}

In \cite{Ivanyi2009} we proved the following assertion.

\begin{theorem} For $n \geq 2$ a\label{theorem-interval} nondecreasing sequence $D = (d_1, d_2, \ldots, d_n)$ of nonnegative
integers is the score sequence of some $(a,b,n)$-tournament if and only if
\begin{equation}
aB_k \leq \sum_{i = 1}^k d_i \leq bB_n - L_k - (n - k)d_k \quad (1 \leq k \leq n),\label{equation-interval}
\end{equation}
where
\begin{equation}
L_0 = 0, \hbox{ and } L_k = \max \left( L_{k - 1}, \ bB_k - \sum_{i = 1}^{k} d_i  \right ) \quad (1 \leq k \leq n).\label{equation-loss}
\end{equation}
\end{theorem}

The theorem proved by Moon \cite{Moon1963}, and later by  Kemnitz and Dolff \cite{Kemnitz1997} for $(a,a,n)$-tournaments is the special
case $a = b$ of Theorem \ref{theorem-interval}. Theorem 3.1.4 of \cite{Frank20081} is the special case $a = b = 2$. The theorem of
Landau \cite{Landau1953} is the special case $a = b = 1$ of Theorem \ref{theorem-interval}.

\subsection{Definition of a testing algorithm}

The following algorithm \textsc{Interval-Test}  decides whether a given $D$ is a
score sequence of an $(a,b,n)$-tournament or not. This algorithm  is based on Theorem
\ref{theorem-interval} and returns $W = \textsc{True}$ if $D$ is a score sequence, and
returns $W = \textsc{False}$ otherwise.

\medskip
\textit{Input}. $a$: minimal number of points divided after each match; \\
$b$: maximal number of points divided after each match.

\textit{Output}. $W$: logical variable $(W = \textsc{True}$ shows that $D$ is an
$(a,b,n)$-tournament.

\textit{Local working variables}. $i$: cycle variable; \\
$L = (L_0,L_1,\ldots,L_n)$: the sequence of the values of the loss function.

\textit{Global working variables}. $n$: the number of players $(n \geq 2)$; \\
$D = (d_1, d_2, \ldots, d_n)$: a nondecreasing sequence of nonnegative integers; \\
$B = (B_0,B_1,\ldots,B_n)$: the sequence of the binomial coefficients; \\
$S = (S_0,S_1,\ldots,S_n)$: the sequence of the sums of the $i$ smallest scores.

\medskip
\noindent \textsc{Interval-Test}$(a,b)$
\vspace{-2mm}
\begin{tabbing}%
199 \= xxx\=xxx\=xxx\=xxx\=xxx\=xxx\=xxx\=xxx \+ \kill
\hspace{-7mm}01 \textbf{for} \= $i \leftarrow 1$ \textbf{to} $n$ \\
\hspace{-7mm}02              \> \textbf{do} \= $L_i \leftarrow \max(L_{i - 1}, \ bB_n - S_i - (n - i)d_i)$ \\
\hspace{-7mm}03              \>             \> \textbf{if} \= $S_i < aB_i$ \\
\hspace{-7mm}04              \>             \>             \> \textbf{then} \= $W \leftarrow \textsc{False}$ \\
\hspace{-7mm}05              \>             \>             \>               \> \textbf{return} $W$ \\
\hspace{-7mm}06              \>             \> \textbf{if} \= $S_i > bB_n - L_i - (n - i)d_i$ \\
\hspace{-7mm}07              \>             \>             \> \textbf{then} \= $W \leftarrow \textsc{False}$ \\
\hspace{-7mm}08              \>             \>             \>               \> \textbf{return} $W$ \\
\hspace{-7mm}09 \textbf{return} $W$
\end{tabbing}

In worst case \textsc{Interval-Test} runs in $\Theta(n)$ time even in the general case $0 < a < b$ (n the best case
the running time of \textsc{Interval-Test} is $\Theta(n)$). It is worth to mention, that the often referenced
Havel--Hakimi algorithm \cite{Hakimi1962,Havel1955} even in the special case $a = b = 1$ decides in $\Theta({n^2})$ time
whether a sequence $D$ is digraphical or not.

\subsection{Definition of an algorithm computing $f$ and $g$}

The following algorithm is based on the bounds of $f$ and $g$ given by Lemma \ref{lemma-fgbounds} and the logarithmic search
algorithm described by D. E. Knuth \cite[page 410]{Knuth1998}.

\medskip
\textit{Input}. No special input (global working variables serve as input).

\textit{Output}. $b$: $f$ (the minimal $F$); \\
$a$: $g$ (the maximal $G$).

\textit{Local working variables}. $i$: cycle variable; \\
$l$: lower bound of the interval of the possible values of $F$; \\
$u$: upper bound of the interval of the possible values of $F$.

\textit{Global working variables}. $n$: the number of players $(n \geq 2)$; \\
$D = (d_1, d_2, \ldots, d_n)$: a nondecreasing sequence of nonnegative integers; \\
$B = (B_0,B_1,\ldots,B_n)$: the sequence of the binomial coefficients; \\
$S = (S_0,S_1,\ldots,S_n)$: the sequence of the sums of the $i$ smallest scores; \\
$W$: logical variable (its value is \textsc{True}, when the investigated $D$ is a score sequence).

\medskip

\noindent \textsc{MinF-MaxG}
\vspace{-2mm}
\begin{tabbing}%
199 \= xxx\=xxx\=xxx\=xxx\=xxx\=xxx\=xxx\=xxx \+ \kill
\hspace{-7mm}01 $B_0 \leftarrow S_0 \leftarrow L_0 \leftarrow 0$ \hspace{2.7cm} $\rhd$ Initialisation \\
\hspace{-7mm}02 \textbf{for} \= $i \leftarrow 1$ \textbf{to} $n$ \\
\hspace{-7mm}03              \> \textbf{do} \= $B_i \leftarrow B_{i - 1} + i - 1$ \\
\hspace{-7mm}04              \>             \> $S_i \leftarrow S_{i - 1} + d_i$  \\
\hspace{-7mm}05 $l \leftarrow \max (\lceil S_n/B_n \rceil,\lceil d_n/(n - 1) \rceil$) \\
\hspace{-7mm}06 $u \leftarrow 2 \left \lceil d_n/(n - 1) \right \rceil$ \\
\hspace{-7mm}07 $W \leftarrow \textsc{True}$                      \hspace{3.9cm} $\rhd$ Computation of $f$ \\
\hspace{-7mm}08 \textsc{Interval-Test}$(0,l)$ \\
\hspace{-7mm}09 \textbf{if} \= $W = \textsc{True}$ \\
\hspace{-7mm}10             \> \textbf{then} \= $b \leftarrow l$ \\
\hspace{-7mm}11             \>               \> \textbf{go to} 21  \\
\hspace{-7mm}12 $b \leftarrow \lceil (l + u)/2 \rceil $ \\
\hspace{-7mm}13 \textsc{Interval-Test}$(0,f)$ \\
\hspace{-7mm}14 \textbf{if} \= $W = \textsc{True}$ \\
\hspace{-7mm}15             \> \textbf{then} \textbf{go to} 17 \\
\hspace{-7mm}16 $l \leftarrow b$ \\
\hspace{-7mm}17 \textbf{if} \= $u = l + 1$ \\
\hspace{-7mm}18             \> \textbf{then} \= $b \leftarrow u$ \\
\hspace{-7mm}19             \>               \> \textbf{go to} 21 \\
\hspace{-7mm}20 \textbf{go to} 14 \\
\hspace{-7mm}21 $l \leftarrow 0$                                    \hspace{5cm} $\rhd$ Computation of $g$ \\
\hspace{-7mm}22 $u \leftarrow f$ \\
\hspace{-7mm}23 \textsc{Interval-Test}$(b,b)$ \\
\hspace{-7mm}24 \textbf{if} \= $W = \textsc{True}$ \\
\hspace{-7mm}25             \> \textbf{then} \= $a \leftarrow f$ \\
\hspace{-7mm}26             \>               \> \textbf{go to} 37 \\
\hspace{-7mm}27 $a \leftarrow \lceil (l + u)/2 \rceil$ \\
\hspace{-7mm}28 \textsc{Interval-Test}$(0,a)$ \\
\hspace{-7mm}29 \textbf{if} \= $W = \textsc{True}$ \\
\hspace{-7mm}30             \> \textbf{then} \= $l \leftarrow a$ \\
\hspace{-7mm}31             \>               \> \textbf{go to} 33 \\
\hspace{-7mm}32 $u \leftarrow a$ \\
\hspace{-7mm}33 \textbf{if} \= $u = l + 1$ \\
\hspace{-7mm}34             \> \textbf{then} \= $a \leftarrow l$ \\
\hspace{-7mm}35             \>               \> \textbf{go to} 37 \\
\hspace{-7mm}36 \textbf{go to} 27 \\
\hspace{-7mm}37 \textbf{return} $a,b$
\end{tabbing}

\textsc{MinF-MaxG} determines $f$ and $g$.

\begin{lemma} Algorithm\label{lemma-fg} \textsc{MinG}-\textsc{MaxG} computes the values $f$ and $g$ for arbitrary sequence
$D = (d_1, d_2,\ldots,d_n)$ in $O(n \log(d_n/(n))$ time.
\end{lemma}
\begin{proof} According to Lemma \ref{lemma-fgbounds} $F$ is an element of the interval
$[\lceil d_n/(n - 1)\rceil, \lceil 2d_n/(n - 1)\rceil]$ and $g$ is an element of the interval $[0,f]$.
Using Theorem B of \cite[page 412]{Knuth1998} we get that $O(\log (d_n/n))$ calls of \textsc{Interval-Test}
is sufficient, so the $O(n)$ run time of  \textsc{Interval-Test} implies the required
running time of \textsc{MinF}-\textsc{MaxG}.
\end{proof}

\subsection{Computing of $f$ and $g$ in linear time}

Analysing Theorem \ref{theorem-interval} and the work of algorithm \textsc{MinF-MaxG} one can observe that
the maximal value of $G$ and the minimal value of $F$ can be computed independently by \textsc{Linear-MinF-MaxG}.

\medskip
\textit{Input}. No special input (global working variables serve as input).

\textit{Output}. $b$: $f$ (the minimal $F$). \\
$a$: $g$ (the maximal $G$).

\textit{Local working variables}. $i$: cycle variable.

\textit{Global working variables}. $n$: the number of players $(n \geq 2)$; \\
$D = (d_1, d_2, \ldots, d_n)$: a nondecreasing sequence of nonnegative integers; \\
$B = (B_0,B_1,\ldots,B_n)$: the sequence of the binomial coefficients; \\
$S = (S_0,S_1,\ldots,S_n)$: the sequence of the sums of the $i$ smallest scores.

\medskip
\noindent \textsc{Linear-MinF-MaxG}
\vspace{-2mm}
\begin{tabbing}%
199 \= xxx\=xxx\=xxx\=xxx\=xxx\=xxx\=xxx\=xxx \+ \kill
\hspace{-7mm}01 $B_0 \leftarrow S_0 \leftarrow L_0 \leftarrow 0 $      \hspace{2.6 cm} $\rhd$ Initialisation \\
\hspace{-7mm}02 \textbf{for} \= $i \leftarrow 1$ \textbf{to} $n$ \\
\hspace{-7mm}03              \> \textbf{do} \= $B_i \leftarrow B_{i - 1} + i - 1$ \\
\hspace{-7mm}04              \>             \> $S_i \leftarrow S_{i - 1} + d_i$  \\
\hspace{-7mm}05 $a \leftarrow 0$ \\
\hspace{-7mm}06 $b \leftarrow \min 2 \left \lceil d_n/(n - 1) \right \rceil$ \\
\hspace{-7mm}07 \textbf{for} \= $i \leftarrow 1$ \textbf{to} $n$         \hspace{3.3cm} $\rhd$ Computation of $g$ \\
\hspace{-7mm}08              \> \textbf{do} \= $a_i \leftarrow \lceil 2S_i/(n^2 - n) \rceil$ \\
\hspace{-7mm}09              \>             \> \textbf{if} \= $a_i > a$ \\
\hspace{-7mm}10              \>             \>             \> \textbf{then} $a \leftarrow a_i$ \\
\hspace{-7mm}11 \textbf{for} \= $i \leftarrow 1 \textbf{ to } n$         \hspace{3.3cm} $\rhd$ Computation of $f$ \\
\hspace{-7mm}12              \> \textbf{do} \= $L_i \leftarrow \max(L_{i - 1},bB_n - S_i - (n - i)d_i)$ \\
\hspace{-7mm}13              \>             \> $b_i \leftarrow (S_i + (n - i)d_i + L_i)/B_i$ \\
\hspace{-7mm}14              \>             \> \textbf{if} \= $b_i < b$ \\
\hspace{-7mm}15              \>             \>             \> \textbf{then} \= $b \leftarrow b_i$ \\
\hspace{-7mm}16 \textbf{return} $a,b$
\end{tabbing}

\begin{lemma} Algorithm\label{lemma-linfg} \textsc{Linear-MinG-MaxG} computes the values $f$ and $g$ for arbitrary sequence
$D = (d_1, d_2,\ldots,d_n)$ in $\Theta(n)$ time.
\end{lemma}
\begin{proof} Lines 01--03, 07, and 18 require only constant time, lines 04--06, 09--12, and 13--17 require $\Theta(n)$ time,
so the total running time is $\Theta(n)$.
\end{proof}

\section{Tournament with $f$ and $g$\label{section-construct}}

The following reconstruction algorithm \textsc{Score-Slicing2} is based on balancing between additional points (they are similar to ,,excess'',
introduced by Brauer et al. \cite{Brauer1968}) and missing points introduced in \cite{Ivanyi2009}. The greediness of
the algorithm Havel--Hakimi \cite{Hakimi1962,Havel1955} also characterises this algorithm.

This algorithm is an extended version of the algorithm \textsc{Score-Slicing} proposed in \cite{Ivanyi2009}.

\subsection{Definition of the minimax reconstruction algorithm}

The work of the slicing program is managed by the following program \textsc{Mini-Max}.

\textit{Input.} No special input (global working variables serve as input).

\textit{Output}. $\mathcal{M} = [1 \ . \ . \ n,1 \ . \ . \ n]$: the point matrix of the reconstructed tournament.

\textit{Local working variables.} $i, \ j$: cycle variables.

\textit{Global working variables.} $n$: the number of players $(n \geq 2)$; \\
$D = (d_1, d_2, \ldots, d_n)$: a nondecreasing sequence of nonnegative integers; \\
$p = (p_0, p_1, \ldots, p_n)$: provisional score sequence; \\
$P = (P_0,P_1,\ldots,P_n)$: the partial sums of the provisional scores; \\
$\mathcal{M}[1 \ . \ . \ n,1 \ . \ . \ n]$: matrix of the provisional points.

\medskip
\noindent \textsc{Mini-Max}
\vspace{-2mm}
\begin{tabbing}%
199 \= xxx\=xxx\=xxx\=xxx\=xxx\=xxx\=xxx\=xxx \+ \kill
\hspace{-7mm}01 \textsc{MinF}-\textsc{MaxG} \hspace{3.2 cm} $\rhd$ Initialisation \\
\hspace{-7mm}02 $p_0 \leftarrow 0$ \\
\hspace{-7mm}03 \textbf{for} \= $i \leftarrow 1$ \textbf{to} $n$ \\
\hspace{-7mm}04              \> \textbf{do} \= \textbf{for} \= $j \leftarrow 1$ \textbf{to} $i - 1$ \\
\hspace{-7mm}05              \>             \>              \> \textbf{do} $\mathcal{M}[i,j] \leftarrow b$ \\
\hspace{-7mm}06              \>             \> \textbf{for} \= $j \leftarrow i$ \textbf{to} $n$ \\
\hspace{-7mm}07              \>             \>              \> \textbf{do} $\mathcal{M}[i,j] \leftarrow 0$ \\
\hspace{-7mm}08              \> $p_i \leftarrow d_i$ \\
\hspace{-7mm}09 \textbf{if} \= $n \geq 3$   \hspace{5.1 cm} $\rhd$ Score slicing for $n \geq 3$ players \\
\hspace{-7mm}10             \> \textbf{then} \= \textbf{for} \= $k\leftarrow n$ \textbf{downto} $3$ \\
\hspace{-7mm}11             \>               \>              \> \textbf{do} \= \textsc{Score-Slicing2}$(k,\mathbf{p}_k,\mathcal{M})$\\
\hspace{-7mm}12 \textbf{if} \= $n = 2$ \hspace{5.1 cm}  $\rhd$ Score slicing for 2 players \\
\hspace{-7mm}13             \> \textbf{then} \= $m_{1,2} \leftarrow p_1$ \\
\hspace{-7mm}14             \>               \> $m_{2,1} \leftarrow p_2$ \\
\hspace{-7mm}15 \textbf{return} $\mathcal{M}$
\end{tabbing}

\subsection{Definition of the score slicing algorithm}

The key part of the reconstruction is the following algorithm \textsc{Score-Slicing2} \cite{Ivanyi2009}.

During the reconstruction process we have to take into account the following bounds:
\begin{equation}
a \leq m_{i,j} + m_{j,i} \leq b \quad (1 \leq i < j \leq n);\label{equation-8}
\end{equation}
\begin{equation}
\mbox{modified scores have to satisfy (\ref{equation-interval});}\label{equation-9}
\end{equation}
\begin{equation}
m_{i,j} \leq p_i \ (1 \leq i, \ j \leq n, i \neq j);\label{equation-10}
\end{equation}
\begin{equation}
\mbox{the monotonicity } p_1 \leq p_2 \leq \ldots \leq p_k \mbox{ has to be saved } \ (1 \leq k \leq n)\label{equation-11}
\end{equation}
\begin{equation}
m_{ii} = 0 \quad (1 \leq i \leq n).
\end{equation}

\textit{Input.} $k$: the number of the actually investigated players $(k > 2)$; \\
$\mathbf{p}_k = (p_0,p_1,p_2,\ldots,p_k) \ (k = 3, \ 4, \ \cdots, \ n)$: prefix of the provisional score sequence $p$; \\
$\mathcal{M}[1 \ . \ . \ n,1 \ . \ . \ n]$: matrix of provisional points.

\textit{Output.} $\mathcal{M}[1 \ . \ . \ n,1 \ . \ . \ n]$: matrix of provisional points; \\
$\mathbf{p}_k = (p_0,p_1,p_2,\ldots,p_k) \ (k = 2, \ 3, \ 4, \ \cdots, \ n - 1)$: prefix of the provisional score sequence $p$.

\textit{Local working variables.} $A = (A_1,A_2,\ldots,A_n)$: the number of the additional points; \\
$M$: missing points (the difference of the number of actual points and the number of maximal possible points of P$_k$);\\
$d$: difference of the maximal decreasable score and the following largest score; \\
$y$: minimal number of sliced points per player; \\
$f$: frequency of the number of maximal values among the scores $p_1, \ p_2,$ \linebreak
\noindent $\ldots, \ p_{k-1}$;\\
$i,  \ j$: cycle variables;\\
$m$: maximal amount of sliceable points; \\
$P = (P_0,P_1,\ldots,P_n)$: the sums of the provisional scores; \\
$x$: the maximal index $i$ with $i < k$ and $m_{i,k} < b$.

\textit{Global working variables}. $n$: the number of players $(n \geq 2)$; \\
$B$ = $(B_0,B_1,B_2,\ldots,B_n)$: the sequence of the binomial coefficients; \\
$a$: minimal number of points divided after each match; \\
$b$: maximal number of points divided after each match.

\medskip
\noindent \textsc{Score-Slicing2}$(k,\mathbf{p}_k,\mathcal{M})$
\vspace{-2mm}
\begin{tabbing}%
199 \= xxx\=xxx\=xxx\=xxx\=xxx\=xxx\=xxx\=xxx \+ \kill
\hspace{-7mm}01 $P_0 \leftarrow 0$     \hspace{3.7 cm}  $\rhd$ Initialisation \\
\hspace{-7mm}02 \textbf{for} \= $i \leftarrow 1$ \textbf{to} $k - 1$  \\
\hspace{-7mm}03              \> \textbf{do} \= $P_i \leftarrow P_{i-1} + p_i$\\
\hspace{-7mm}04              \>             \> $A_i \leftarrow P_i - aB_i$ \\
\hspace{-7mm}05 $M \leftarrow (k - 1)b - p_k$\\
\hspace{-7mm}06 \textbf{while} \= $M > 0$ \textbf{and} $A_{k - 1} > 0$  \hspace{0.2 cm} $\rhd$ There are missing and additional points\\
\hspace{-7mm}07                \> \textbf{do} \= $x\leftarrow k-1$\\
\hspace{-7mm}08                \>             \> \textbf{while} \= $r_{x,k} = b$ \\
\hspace{-7mm}09                \>             \>                \> \textbf{do} \= $x\leftarrow x-1$\\
\hspace{-7mm}10                \>             \> $f \leftarrow 1$ \\
\hspace{-7mm}11                \>             \> \textbf{while} \= $p_{x-f +1}=p_{x-f}$\\
\hspace{-7mm}12                \>             \>                \> \textbf{do} $f = f + 1$\\
\hspace{-7mm}13                \>             \> $d\leftarrow p_{x- f + 1}-p_{x- f}$\\
\hspace{-7mm}14                \>             \> $m\leftarrow \min(b,d,\lceil A_x/f \rceil,\lceil M/f \rceil)$\\
\hspace{-7mm}15                \>             \> \textbf{for} \= $i \leftarrow f$ \textbf{downto} $1$\\
\hspace{-7mm}16                \>             \>              \> \textbf{do} \= $y \leftarrow \min(b-m_{x+1-i,k},m,M,A_{x+1-i},p_{x+1-i})$\\
\hspace{-7mm}17                \>             \>              \>             \> $m_{x+1-i,k} \leftarrow m_{x+1-i,k} + y$\\
\hspace{-7mm}18                \>             \>              \>             \> $p_{x+1-i} \leftarrow p_{x+1-i}-y$\\
\hspace{-7mm}19                \>             \>              \>             \> $m_{k,x+1-i} \leftarrow m_{k,x+1-i} - m_{x+1-i,k}$\\
\hspace{-7mm}20                \>             \>              \>             \> $M \leftarrow M - y$\\
\hspace{-7mm}21                \>             \> \textbf{for} \= $j \leftarrow i$ \textbf{downto} $1$\\
\hspace{-7mm}22                \>             \>              \> $A_{x+1-i} \leftarrow A_{x+1-i} - y$ \\
\hspace{-7mm}23 \textbf{while} \= $M > 0$ and $A_{k-1} = 0$                              \hspace{0.8 cm} $\rhd$ No additional points \\
\hspace{-7mm}24                \> \textbf{do} \= \textbf{for} \= $i \leftarrow k - 1$ \textbf{downto} $1$  \\
\hspace{-7mm}25                \>             \>              \= $y \min(m_{k,i},M,m_{k,i+m_{i,k} - a})$ \\
\hspace{-7mm}26                \>             \>              \> $m_{ki} \leftarrow m_{k,i} - y$ \\
\hspace{-7mm}27                \>             \>              \> $M \leftarrow M - y$ \\
\hspace{-7mm}28 \textbf{return} $\mathbf{p}_k,\mathcal{M}$
\end{tabbing}

Let's consider an example. Figure \ref{figure-2} shows the point table of a
$(2,10,6)$-tournament $T$.

\begin{figure}[!h]
\begin{center}
\begin{tabular}{|c|c|c|c|c|c|c|c|}  \hline
Player/Player &  P$_1$ & P$_2$ & P$_3$ & P$_4$ & P$_5$  & P$_6 $ & Score \\ \hline
P$_1 $ & --- & 1 &  5 & 1 &  1 & 1 & 9  \\ \hline
P$_2$ & 1 & --- &  4 & 2  & 0 & 2 & 9  \\ \hline
P$_3$ & 3 & 3 &  --- & 5 & 4 & 4 & 19  \\ \hline
P$_4$ & 8 & 2 & 5 & --- & 2 & 3 & 20  \\ \hline
P$_5$ & 9 & 9 & 5 & 7 & --- & 2 & 32   \\ \hline
P$_6$ & 8 & 7 & 5 & 6 & 8 & --- & 34  \\ \hline
\end{tabular} \\
\caption{The point table of a $(2,10,6)$-tournament $T$.\label{figure-2}}
\end{center}
\vspace{-2mm}
\end{figure}

The score sequence of $T$ is $D$ = (9,9,19,20,32,34).
In \cite{Ivanyi2009} the algorithm \textsc{Score-Slicing2} resulted the
point table represented in Figure \ref{figure-3}.

\begin{figure}[h]
\begin{center}
\vspace{2mm}
\begin{tabular}{|c|c|c|c|c|c|c|c|}  \hline
Player/Player &  P$_1$ & P$_2$ & P$_3 $ & P$_4 $ & P$_5$  & P$_6$ & Score \\ \hline
P$_1$ & --- & 1 &  1 & 6 &  1 & 0 & 9  \\ \hline
P$_2$ & 1 & --- &  1 & 6  & 1 & 0 & 9  \\ \hline
P$_3$ & 1 & 1 &  --- & 6 & 8 & 3 & 19  \\ \hline
P$_4$ & 3 & 3 & 3 & --- & 8 & 3 & 20  \\ \hline
P$_5$ & 9 & 9 & 2 & 2 & --- & 10 & 32   \\ \hline
P$_6$ & 10 & 10 & 7 & 7 & 0 & --- & 34  \\ \hline
\end{tabular} \\

\vspace{2 mm}
\caption{The point table of $T$ reconstructed by \textsc{Score-Slicing2}.\label{figure-3}}
\end{center}
\vspace{-2mm}
\end{figure}

The algorithm \textsc{Mini-Max} starts with the computation of $f$. \textsc{MinF-MaxG} called in line 01 begins with initialisation,
including provisional setting of the elements of $\mathcal{M}$ so, that $m_{ij} = b$, if $i > j$, and $m_{ij} = 0$ otherwise.
Then \textsc{MinF-MaxG} sets the lower bound $l = \max(9,7) = 9$ of $f$ in line 05 and tests it in line 08 by
\textsc{Interval-Test}. The test shows that $l = 9$ is large enough so \textsc{Mini-Max} sets $b = 9$ in line 12 and jumps to
line 21 and begins to compute $g$. \textsc{Interval-Test} called in line 23 shows that $a = 9$ is too large, therefore
\textsc{MinF-MaxG} continues with the test of $a = 5$ in line 27. The result is positive, therefore comes the test of $a = 7$,
then the test of $a = 8$. Now $u = l + 1$ in line 33, so $a = 8$ is fixed, and the control returns to line 02 of \textsc{Mini-Max}.

Lines 02--08 contain initialisation, and \textsc{Mini-Max} begins the reconstruction of a $(8,9,6)$-tournament in line 9. The
basic idea is that \textsc{Mini-Max} successively determines the won and lost points of P$_6$, P$_5$, P$_4$ and P$_3$ by repeated
calls of \textsc{Score-Slicing2} in line 11, and finally it computes directly the result of the match between P$_2$ and P$_1$
in lines 12--14.

At first \textsc{Mini-Max} computes the results of P$_6$ calling \textsc{Score-Slicing2} with parameter $k = 6$.
The number of additional points of the first five players is
$A_5 = 89 - 8 \cdot 10 = 9$ according to line 04, the number of missing points of P$_6$ is $M = 5 \cdot 9 - 34 = 11$
according to line 05. Then \textsc{Score-Slicing2}  determines
the number of maximal numbers among the provisional scores $p_1, \ p_2, \ \ldots, \ p_5$ ($f = 1$ according to lines 10--12) and
computes the difference between $p_5$ and $p_4$ ($d = 12$ according to line 13). In line 14 we get, that $m = 9$ points are
sliceable, and P$_5$ gets these points in the match with P$_6$ in line 17, so the number of missing points of P$_6$
decreases to $M = 11 - 9 = 2$ (line 20) and the number of additional point decreases to $A_5 = 9 - 9 = 0$.
Therefore the computation continues in lines 23--28 and $m_{64}$ and $m_{63}$ will be decreased by 1 resulting $m_{64} = 8$
and $m_{63} = 8$ as the seventh line and seventh column of Figure \ref{figure-4} show.
The returned score sequence is $\mathbf{p}_5 = (9,9,19,20,23)$.

\begin{figure}[!h]
\begin{center}
\vspace{2mm}
\begin{tabular}{|c|c|c|c|c|c|c|c|}  \hline
Player/Player &  P$_1$ & P$_2$ & P$_3 $ & P$_4$ & P$_5$  & P$_6 $ & Score \\ \hline
P$_1$ & --- &  4  & 4   &  1  &  0  &  0  &  9  \\ \hline
P$_2$ &  4  & --- & 4   &  1  &  0  &  0  &  9  \\ \hline
P$_3$ &  4  &  4  & --- &  7  &  4  &  0  & 19  \\ \hline
P$_4$ &  7  &  7  &  1  & --- &  5  &  0  & 20  \\ \hline
P$_5$ &  8  &  8  &  4  &  3  & --- &  9  & 32   \\ \hline
P$_6$ &  9  &  9  &  8  &  8  &  0  & --- & 34  \\ \hline
\end{tabular} \\

\vspace{2 mm}
\caption{The point table of $T$ reconstructed by \textsc{Mini-Max}.\label{figure-4}}
\end{center}
\vspace{-2mm}
\end{figure}

Second time \textsc{Mini-Max} calls \textsc{Score-Slicing2} with parameter $k = 5$, and get $A_4 = 9$ and $M = 13$. At first $P_4$
gets $1$ point, then $P_3$ and $P_4$ get both 4 points, reducing $M$ to $4$ and $A_4$ to $0$. The computation continues
in line 23 and results the further decrease of $m_{54}$, $m_{53}$, $m_{52}$, and $m_{51}$ by 1, resulting $m_{54} = 3$,
$m_{53} = 4$,  $m_{52} = 8$, and $m_{51} = 8$ as the sixth row of Figure \ref{figure-4} shows. The returned score sequence is
$\mathbf{p_4} = (9,9,15,15)$

Third time \textsc{Mini-Max} calls \textsc{Score-Slicing2} with parameter $k = 4$, and get $A_3 = 11$ and $M = 11$. At first
P$_3$ gets $6$ points, then P$_3$ further 1 point, and P$_2$ and P$_1$ also both get 1 point, resulting $m_{34} = 7$, $m_{43} = 2$,
$m_{42} = 8$, $m_{24} = 1$, $m_{14} = 1$ and $m_{14} = 8$, further $A_3 = 0$ and $M = 2$. The computation continues in lines
23--28 and results a decrease of $m_{43}$ by 1 point resulting $m_{43} = 1$, $m_{42} = 7$, and $m_{41} = 7$, as the fifth row
and fifth column of Figure \ref{figure-4} show. The returned score sequence is $\mathbf{p}_3 = (8,8,8)$.

Fourth time \textsc{Mini-Max} calls \textsc{Score-Slicing2} with parameter $k = 3$, and gets $A_2 = 8$ and $M = 10$. At first
P$_1$ and P$_2$ get $4$ points, resulting $m_{13} = 4$, and $m_{23} = 4$, and $M = 2$, and $A_2 = 0$. Then \textsc{Mini-Max}
sets in lines 23--26 $m_{31} = 4$ and $m_{32} = 4$. The returned score sequence is  $\mathbf{p}_2 = (4,4)$.

Finally \textsc{Mini-Max} sets $m_{12} = 4$ and $m_{21} = 4$ in lines 14--15 and returns the point matrix
represented in Figure \ref{figure-4}.

The comparison of Figures \ref{figure-3} and \ref{figure-4} shows a large difference between the simple reconstruction of
\textsc{Score-Slicing2} and the minimax reconstruction of \textsc{Mini-Max}: while in the first case the maximal value of
$m_{ij} + m_{ji}$ is $10$ and the minimal value is $2$, in the second case the maximum equals to $9$ and the minimum equals
to $8$, that is the result is more balanced (the given $D$ does not allow to build a perfectly balanced $(k,k,n)$-tournament).

\subsection{Analysis of the minimax reconstruction algorithm}

The main result of this paper is the following assertion.

\begin{theorem} If $n \geq 2$ is a positive integer and $D = (d_1, d_2, \ldots, d_n)$ is a nondecreasing sequence of nonnegative
integers, then  there exist positive integers $f$ and $g$, and a $(g,f,n)$-tournament $T$ with point matrix $\mathcal{M}$ such, that
\begin{equation}
f = \min(m_{ij} + m_{ji}) \leq b,
\end{equation}
\vspace{-4mm}
\begin{equation}
g = \max m_{ij} + m_{ji} \geq a
\end{equation}
for any $(a,b,n)$-tournament, and algorithm \textsc{Linear-MinF-MaxG} computes $f$ and $g$ in $\Theta(n)$ time,
and algorithm \textsc{Mini-Max} generates a suitable $T$ in $O(d_n n^2)$ time.
\end{theorem}
\begin{proof} The correctness of the algorithms \textsc{Score-Slicing2}, \textsc{Min}$F$-\textsc{Max}$G$ implies the
correctness of \textsc{Mini-Max}.

Lines 1--46 of \textsc{Mini-Max} require $O(\log (d_n/n))$ uses of \textsc{MinG-MaxF}, and one search needs $O(n)$ steps
for the testing, so the computation of $f$ and $g$ can be executed in
$O(n \log (d_n/n))$ times.

The reconstruction part (lines 47--55) uses algorithm \textsc{Score-Slicing2}, which runs in $O(b n^3)$ time \cite{Ivanyi2009}.
\textsc{Mini-Max} calls \textsc{Score-Slicing2} $n - 2$ times with $f \leq 2\lceil d_n/n \rceil$, so $n^3 d_n/n = d_n n^2$
finishes the proof.
\end{proof}

The property of the tournament reconstruction problem that the extremal values of $f$ and $g$ can be determined independently
and so there exists a tournament $T$ having both extremal features is called linking property. This concept was introduced
by Ford and Fulkerson in 1962 \cite{Ford1962} and later extended by A. Frank in \cite{Frank20081}.

\section{Summary}

A nondecreasing sequence of nonnegative integers $D = (d_1, d_2, \ldots,d_n)$ is a score sequence of a $(1,1,1)$-tournament, iff
the sum of the elements of $D$ equals to $B_n$ and the sum of the first $i \ (i = 1, \ 2, \ \ldots, \ n - 1)$ elements of
$D$ is at least $B_i$ \cite{Landau1953}.

$D$ is a score sequence of a $(k,k,n)$-tournament, iff the sum of the elements of $D$ equals to $kB_n$, and the sum of the first
$i$ elements of $D$ is at least $kB_i$ \cite{Kemnitz1997,Moon1962}.

$D$ is a score sequence of an $(a,b,n)$-tournament, iff (\ref{equation-interval}) holds \cite{Ivanyi2009}.

In all 3 cases the decision whether $D$ is digraphical requires only linear time.

In this paper the results of \cite{Ivanyi2009} are extended proving that for any $D$ there exists an optimal
minimax realization $T$, that is a tournament having $D$ as its out-degree sequence, and maximal $G$, and minimal $F$
in the set of all realizations of $D$.

In  a continuation \cite{Ivanyi2010} of this paper we construct balanced as possible tournaments
in a similar way if not only the out-degree sequence but the in-degree sequence is also given.

\subsection*{Acknowledgement}

 The author thanks Andr\'as Frank
(E\"otv\"os Lor\'and University) for valuable advises concerning
the application of flow theory, P\'eter L. Erd\H os (Alfr\'ed
R\'enyi Institute of Mathematics of HAS) for the consultation and
the unknown referee proposing corrections of the manuscript.

The research was supported by the project T\'AMOP-4.2.1.B-09/1/KMR--2010-003 of E\"otv\"os Lor\'and University.

\bigskip
\rightline{\emph{Received: March 5, 2010; Revised: April 18, 2010}}     

\end{document}